\documentclass[10pt,twoside,reqno]{amsart}
\usepackage[foot]{amsaddr}
\usepackage{amssymb,amsfonts,amsmath,amsthm}
\usepackage{calrsfs}
\usepackage[mathscr]{eucal}
\usepackage{bbm}
\usepackage{fullpage}
\usepackage{blkarray}
\usepackage{array}
\usepackage{longtable}

\theoremstyle{definition}
\newtheorem{thm}{Theorem}
\newtheorem*{thm*}{Theorem}
 
\newtheorem{lem}[thm]{Lemma}

\newtheorem{rem}[thm]{Remark}

\numberwithin{equation}{section}
\numberwithin{thm}{section}

\usepackage{enumerate}
\usepackage[shortlabels]{enumitem}

\usepackage{hyperref}
\usepackage[dvipsnames]{xcolor}
\newcommand\myshade{85}
\colorlet{mylinkcolor}{red}
\colorlet{mycitecolor}{blue}
\colorlet{myurlcolor}{Aquamarine}
\hypersetup{
	linkcolor  = mylinkcolor,
	citecolor  = mycitecolor,
	urlcolor   = myurlcolor!\myshade!black,
	colorlinks = true,
}

\usepackage{tikz}

\usetikzlibrary{arrows}

\newcommand{\la}[1]{\mathfrak{#1}}

\newcommand{\ZZ}{\mathbb{Z}}

\newcommand{\ch}{\mathsf{ch}}

\newcommand{\fin}{\mathrm{final}}

\allowdisplaybreaks

\begin{document}

\date{}

\title{On $q$-series for principal characters of standard $A_2^{(2)}$-modules}
\author{Shashank Kanade}
\address{University of Denver} 
\email{\texttt{shashank.kanade@du.edu}}
\thanks{S.K. is currently supported by Simon's Collaboration Grant for Mathematicians, \#636937.}
\author{Matthew C.\ Russell}
\address{Rutgers, The State University of New Jersey} 
\email{\texttt{russell2@math.rutgers.edu}}
\thanks{
	We thank Chris Jennings-Shaffer and Jeremy Lovejoy for their very helpful comments and suggestions.
	We are very grateful to Drew Sills 
	for his interest.}

\begin{abstract}
We present sum-sides for principal characters of all standard (i.e., integrable and highest-weight) irreducible modules for the affine Lie algebra
$A_2^{(2)}$. We use modifications of five known Bailey pairs; three of these are sufficient to obtain all the necessary principal characters. We then use the technique of Bailey lattice appropriately extended to include ``out-of-bounds'' values of one of the parameters, namely, $i$.
We demonstrate how the sum-sides break into six families depending on the level of the modules modulo 6, confirming a conjecture of McLaughlin--Sills.
\end{abstract}
\maketitle

\section{Introduction}

Let $\la{g}$ be an affine Kac-Moody Lie algebra of rank $r+1$ ($r\geq 1$) and let $L(\lambda)$ be a standard  (i.e., integrable and highest weight), irreducible module for $\la{g}$ 
of level $\ell\in\mathbb{Z}_{\geq 0}$ and highest weight $\lambda$ (see \cite{Kac-book} for all the relevant definitions).
Such $L(\lambda)$ are direct sums of their finite-dimensional weight spaces.
Thus, to $L(\lambda)$ we associate the generating function of the dimensions of the weight spaces, thereby obtaining what is known as the character of $L(\lambda)$ denoted by $\ch(L(\lambda))$. This character naturally lives in the space $e^\lambda\ZZ[[x_0,x_1,\cdots, x_r]]$.
Throughout this paper, we will omit the factor $e^\lambda$, in effect normalizing the characters.
Specializing all variables $x_0,\dots,x_r\mapsto q$ where $q$ is a formal variable is known the principal specialization of the character. 
We will denote principally specialized characters by $\chi$:
\begin{align}
\chi(\bullet) = \ch(\bullet)|_{x_0,\dots,x_r\mapsto q}.
\end{align}
By Weyl-Kac character formula and Lepowsky's numerator formula, it is straightforward to see that principally specialized characters of standard modules are infinite periodic products (see below for  examples).

In this paper, we are concerned with a certain substructure of $L(\lambda)$, namely, the space of highest weight vectors  of $L(\lambda)$ with respect to the principal Heisenberg subalgebra of $\la{g}$. This is called the vacuum space and we denote it by $\Omega(L(\lambda))$. The vacuum space also has a character, and upon principal specialization, it turns into an infinite periodic product.
We will call this the principal character of $L(\lambda)$. 
For instance, with
 $\la{g}=A_1^{(1)}$ (also known as $\widehat{\la{sl}}_2$), 
and the level $3$ standard modules $L(3\Lambda_0), L(2\Lambda_0+\Lambda_1),L(\Lambda_0+2\Lambda_1),L(3\Lambda_1)$, we have:
\begin{align}
\chi(L(3\Lambda_0)) = \chi(L(3\Lambda_1)) = \frac{1}{(q;q^2)_{\infty}}\frac{1}{(q^2,q^3;q^5)_\infty},\nonumber\\
\chi(L(2\Lambda_0+\Lambda_1)) = \chi(L(\Lambda_0+2\Lambda_1))= \frac{1}{(q;q^2)_{\infty}}\frac{1}{(q^1,q^4;q^5)_\infty}.
\end{align}
But more importantly for us,
\begin{align}
\chi(\Omega(3\Lambda_0)) = \chi(\Omega(3\Lambda_1)) =\frac{1}{(q^2,q^3;q^5)_\infty},\quad
\chi(\Omega(2\Lambda_0+\Lambda_1)) = \chi(\Omega(\Lambda_0+2\Lambda_1)) = \frac{1}{(q,q^4;q^5)_\infty}.
\end{align}
Here and throughout we use:
\begin{align}
(a_1,a_2\cdots, a_k;\,\,q)_n=\prod_{0\leq t < n} (1-a_1q^t)\cdots (1-a_kq^t).
\end{align}

The products in the last two equations are the ``product-sides'' of the classical Rogers-Ramanujan identities:
\begin{align}
\sum_{n\geq 0}\frac{q^{n^2}}{(q;q)_n} = \frac{1}{(q,q^4;q^5)_\infty}, \quad\quad
\sum_{n\geq 0}\frac{q^{n^2+n}}{(q;q)_n} = \frac{1}{(q^2,q^3;q^5)_\infty} .\label{eqn:rr}
\end{align}
The problem of finding combinatorial interpretations (especially using representation-theoretic techniques) for these sum-sides is an \emph{extremely} important one. We shall not address this problem here, instead referring the reader to \cite{Lep-survey}, \cite{LepWil-1}, \cite{LepWil-2}, \cite{MeuPri}, \cite{Sil-book}, etc. 

Sum-sides such as in \eqref{eqn:rr} are known for \emph{all} standard $\widehat{\la{sl}}_2$ modules. For odd levels, these give rise to the Gordon-Andrews identities, and at even levels, the Andrews-Bressoud identities.
In this paper, we provide sum-sides for \emph{all} standard modules of the next ``simplest'' affine Lie algebra, namely, $A_2^{(2)}$.
For some very recent developments regarding levels $4$, $5$ and $7$ we refer the reader to \cite{TakTsu-nandi} and \cite{TakTsu-A22A132}.

We use the technique of Bailey pairs and the Bailey lattice to provide the sum-sides.
Previously, sum-sides for certain modules at every level were given in \cite{Sil-capparelli}, \cite{McLSil-1824}, \cite{McLSil-comb} and \cite{Sil-a22bailey}; we are especially motivated by these papers. While it didn't specifically mention $A_2^{(2)}$, the paper \cite{And-multiRR} is relevant as well.
Especially, in \cite{McLSil-1824},  McLaughlin and Sills conjectured that integrable modules for $A_2^{(2)}$ would break up into six families (based on level modulo $6$), each explained by one Bailey pair. 
They also showed how to get principal character for (at least) one module at every level using these techniques. 
We show that the McLaughlin-Sills conjecture is true, except  we show that three Bailey pairs suffice to give the characters for the six families. 
We actually analyse a total of five Bailey pairs and for each level that is not divisible by $3$, each module is explained separately by two different Bailey pairs. 
This is exhibited in Tables \ref{table:pairs} and \ref{table:ids} below.

Certain important modifications of known Bailey pairs and an extension of the Bailey lattice machinery have proved \emph{immensely} useful for us.
The uniformity with which these produce \emph{all} requisite $A_2^{(2)}$ principal characters is remarkable and does not seem to be explored before in this context. We now explain these modifications.

Recall the definition of Bailey pair: A pair of $q$-series sequences $\alpha_n,\beta_n$ ($n\in\ZZ_{\geq 0}$) forms a Bailey pair with respect to a base $a$ if they satisfy the condition \eqref{eqn:bailey}. 
A Bailey pair gives rise to identities, one for each $k\geq 1, 0\leq i\leq k$ upon using the Bailey lattice \cite[Thm.\ 3.1]{AgaAndBre-lattice}. We shall need to specialize \cite[Thm.\ 3.1]{AgaAndBre-lattice} in various ways.
For us, the 
level $\ell$ of modules in a given family grows roughly as $6k$.

First modification we use is as follows. We start with Bailey pairs A1, A2, A6, A7 of \cite{Sla-pairs} and P1 of \cite{McLSil-1824}, we retain the corresponding $\beta_n$'s, but find \emph{new} $\alpha_n$ with respect to the base shifted as $a\mapsto aq$. This is achieved via the special case $k=1, d_1\rightarrow 0$ of \cite[Thm.\ 2.3]{Lov-lattice}.  Not only do we find that the modified pairs have much more elegant formulas (granting notation \eqref{eqn:alphatilde}) in comparison to original pairs, (see our Table \ref{table:pairs} vis-\`{a}-vis \cite{Sla-pairs}, \cite{McLSil-1824}), but that these modifications work very well with our extension of the Bailey lattice machinery to produce exactly the required characters.

Secondly, for each $k$, the bailey lattice as given in \cite{AgaAndBre-lattice} produces $k+1$ identities, however, the number of modules at corresponding levels is roughly $3k$. To account for these additional characters, we need to extend the Bailey lattice to $i>k$, we do so mainly by incorporating ``backward moves'', see 
\eqref{eqn:b1}, \eqref{eqn:b2} below.

In 2014, Nandi conjectured intriguing combinatorial sum-sides for the level $4$ principal characters \cite{Nan-thesis} and in a significant recent development \cite{TakTsu-nandi} these identities were proved. The final step in the proof of \cite{TakTsu-nandi} uses Slater's mod-14 identities, \cite{Sla-ids}. We hope that the identities we present are similarly useful for other levels.

Recall that $A_1^{(1)}$ gave rise to two families of principal characters (based on levels mod $2$) and similarly, $A_2^{(2)}$ is now shown to give rise to six families. 
It is has not escaped our notice that in order to construct the principally specialized modules of $A_1^{(1)}$, one uses an automorphism of $A_1 = \la{sl}_2$ which has order two and for $A_2^{(2)}$ one uses the twisted Coxeter automorphism of $A_2 = \la{sl}_3$ which has order six \cite{Fig-thesis}. It will be interesting to see if this phenomena persists at higher ranks.

As we remarked earlier, the problem of finding combinatorial interpretations of the sum-sides is important. We shall address it elsewhere.

\section{Slater and Bailey}

Two sequences $\alpha_n$, $\beta_n$ ($n\geq 0$) form a Bailey pair (\cite{Sla-pairs}, \cite{And-qserbook}) with respect to base $a$ if for all $n\geq 0$:
\begin{equation}
\beta_n=\sum_{t=0}^n\frac{\alpha_n}{(q)_{n-t}(aq)_{n+t}}.
\label{eqn:bailey}
\end{equation}
It will be very convenient for us to use the notation:
\begin{align}
\widetilde{\alpha}_n=\frac{1-a}{1-aq^{2n}}\alpha_n.\label{eqn:alphatilde}
\end{align} 

Several ways of getting new Bailey pairs out of a given one are important to us: base shift, forward moves, backward moves (two of each kind), 
base changes (again, two types). 
The forward moves are just special limits of the actual Bailey lemma. Base shift and base change are two different operations.

We start by describing the base shift. We will use this operation on well-known Bailey pairs to arrive at pairs more suitable for our purposes.
\begin{lem}(Base shift) \label{lem:baseshift}
Suppose $\alpha_n',\beta_n$ is a Bailey pair with respect to the base $a$. Then, $\alpha_n,\beta_n$ is a Bailey pair with respect to base $aq$ iff
\begin{align}
\widetilde{\alpha}_n=
\sum_{r=0}^n 
a^{n-r}q^{n^2-r^2}\alpha_r'.\label{eqn:baseshiftalphatilde}
\end{align}
Moreover, \eqref{eqn:baseshiftalphatilde} holds iff:
\begin{align}
\widetilde{\alpha}_0=\alpha'_0,\quad \widetilde{\alpha}_{n+1} = aq^{2n+1}\cdot \widetilde{\alpha}_{n} + \alpha'_{n+1}\,\,(n\geq 0).
\label{eqn:baseshiftalphatilde_rec}
\end{align}
\end{lem}
\begin{proof}
Equation \eqref{eqn:baseshiftalphatilde} is the special $k=1, d_1\rightarrow 0$ of \cite[Thm.\ 2.3]{Lov-lattice}. 
Equations \eqref{eqn:baseshiftalphatilde} and \eqref{eqn:baseshiftalphatilde_rec} are clearly equivalent.
\end{proof}

\begin{lem}
If $\alpha_n$, $\beta_n$ is a Bailey pair with respect to $a$, then so are the following.
Here, ``F'' stands for forward and ``B'' for backward moves.
\begin{align}
\tag{F1}
\label{eqn:f1}
\beta_n' = \sum_{j=0}^n\frac{a^jq^{j^2}}{(q)_{n-j}}\beta_j,\quad
\alpha_n'&= a^nq^{n^2}\alpha_n.
\end{align}
\begin{align}
\tag{B1}
\label{eqn:b1}
\beta_n' = (-1)^na^{-n}q^{-n^2}\sum_{j=0}^n(-1)^{j}\frac{q^{\binom{n-j}{2}}}{(q)_{n-j}}\beta_j,\quad
\alpha_n'&= a^{-n}q^{-n^2}\alpha_n.
\end{align}
\begin{align}
\tag{F2}
\label{eqn:f2}
\beta_n' = \sum_{j=0}^n\frac{(-q)_jq^{\binom{j}{2}} a^j}{(q)_{n-j}(-a)_n}\beta_j,\quad
\alpha_n'&= \frac{(-q)_nq^{\binom{n}{2}}a^n}{(-a)_n}\alpha_n.
\end{align}
\begin{align}
\tag{B2}
\label{eqn:b2}
\beta_n' = \frac{(-1)^na^{-n}q^{-\binom{n}{2}}}{(-q)_n}\sum_{j=0}^n(-1)^{j}\frac{q^{\binom{n-j}{2}}(-a)_j}{(q)_{n-j}}\beta_j,\quad
\alpha_n'&= \frac{(-a)_nq^{-\binom{n}{2}}a^{-n}}{(-q)_n}\alpha_n.
\end{align}
\end{lem}	
\begin{proof}
For \eqref{eqn:f1}, take $\rho,\sigma\rightarrow\infty$ in \cite[Lem.\ 1.1]{AgaAndBre-lattice}.
For \eqref{eqn:b1}, take $\rho,\sigma\rightarrow\infty$ in \cite[Lem.\ 1.1]{AgaAndBre-lattice}.
For \eqref{eqn:f2}, take $\rho\rightarrow-q,\sigma\rightarrow\infty$ in \cite[Lem.\ 1.1]{AgaAndBre-lattice}.
For \eqref{eqn:b2}, take $\rho\rightarrow-q,\sigma\rightarrow\infty$ in \cite[Lem.\ 1.1]{AgaAndBre-lattice}.
\end{proof}

\begin{lem}
If $\alpha_n$, $\beta_n$ is a Bailey pair with respect to $a$, then the following are Bailey pairs with respect to $aq^{-1}$.
These are called the base change moves.
\begin{align}
\tag{BC1}
\label{eqn:bc1}
\beta_n' = \sum_{j=0}^n\frac{a^jq^{j^2-j}}{(q)_{n-j}}\beta_j,\quad
\alpha_n'&= a^nq^{n^2-n}\left\lbrace
\widetilde{\alpha}_n-
aq^{2n-2}\widetilde{\alpha}_{n-1}
\right\rbrace.
\end{align}
\begin{align}
\tag{BC2}
\label{eqn:bc2}
\beta_n' = \sum_{j=0}^n&\frac{(-q)_j}{(-a/q)_n} \frac{a^jq^{\binom{j}{2} - j}}{(q)_{n-j}}\beta_j,\quad
\alpha_n'= \frac{(-q)_na^nq^{\binom{n}{2}-n}}{(-a/q)_n}
\left\lbrace
\widetilde{\alpha}_n-
aq^{2n-2}\widetilde{\alpha}_{n-1}
\right\rbrace.\nonumber
\end{align}
In both cases, we take:
\begin{align}
\alpha_0'=\alpha_0.
\end{align}
\end{lem}
\begin{proof}
For \eqref{eqn:bc1}, take $\rho,\sigma\rightarrow\infty$ in \cite[Lem.\ 1.2]{AgaAndBre-lattice}.
For \eqref{eqn:bc2}, take $\rho\rightarrow-q,\sigma\rightarrow\infty$ in \cite[Lem.\ 1.2]{AgaAndBre-lattice}.
\end{proof}

\section{Locating identities}\label{sec:gettingids}

\subsection{General Strategy}
Starting with a Bailey pair $\alpha_n, \beta_n$ with respect to certain base $a$,
and given $k,i$ with $k\geq 1$, $i\geq 0$, we shall use various sequences of moves (depending on $k,i$) to arrive
at our ultimate Bailey pair $\alpha_n^\fin, \beta_n^\fin$, 
either with base $a$ or $aq^{-1}$ (depending on whether we use base change).
We then write down the equation \eqref{eqn:bailey} that asserts that this ultimate Bailey pair is indeed a Bailey pair, and let $n\rightarrow\infty$ in this equation.
This is essentially our intended identity, up to some factors of $(q)_{\infty}$, $1-q$, $1-q^2$, etc.
The side involving $\beta^\fin$ is our sum-side.
For suitable choices of $\alpha$ in the initial Bailey pair, the final side involving $\alpha^\fin$s turns into a product upon invoking the Quintuple Product Identity (QTPI):
\begin{align}
Q&(s,t)=\prod_{n\geq 1}(1-s^n)(1-s^nt)(1-s^{n-1}t^{-1})(1-s^{2n-1}t^2)(1-s^{2n-1}t^{-2})\label{eqn:QTPI}\\
&=\sum_{n\in\ZZ}s^{(3n^2+n)/2}(t^{3n}-t^{-3n-1}) \label{eqn:QTPI-I}\\
&=\sum_{n\geq 0}s^{(3n^2+n)/2}t^{3n} 
+ \sum_{n\geq 1}s^{(3n^2-n)/2}t^{-3n}
- \sum_{n\geq 0}s^{(3n^2+n)/2}t^{-3n-1} 
- \sum_{n\geq 1}s^{(3n^2-n)/2}t^{3n-1} 
\label{eqn:QTPI-II}\\
&=\sum_{n\geq 0}s^{(3n^2-n)/2}t^{-3n} 
+ \sum_{n\geq 0}s^{(3n^2+7n+4)/2}t^{3n+3}
- \sum_{n\geq 0}s^{(3n^2+n)/2}t^{-3n-1} 
- \sum_{n\geq 0}s^{(3n^2+5n+2)/2}t^{3n+2}. 
\label{eqn:QTPI-III}
\end{align}

We now describe the sequences of moves that we will use. In each case, we provide the final form of the identity.

\subsection{No forward moves of type two}\label{sec:lim1}
\begin{enumerate}
\item If $k-i\geq 0$, use \eqref{eqn:f1} $k-i$ times. 
If $k-i< 0$, run  \eqref{eqn:b1} $|k-i|$ times.
\item If $i\geq 1$, use \eqref{eqn:bc1}.
\item If $i\geq 1$, use \eqref{eqn:f1} $i-1$ times, noting that the base is $aq^{-1}$.
\end{enumerate}		
The formula for $\alpha_n^\fin$ depends on whether $i=0$ or not. 
If $i=0$, we have:
\begin{align}
\alpha_n^\fin=a^{kn}q^{kn^2}\alpha_n.
\end{align}
If $i\geq 1$, it is not hard to see that we have:
\begin{align}
\alpha_n^\fin=  
a^{in}q^{in^2-in}
\left\lbrace a^{(k-i)n}q^{(k-i)n^2}\widetilde{\alpha}_n-a^{(k-i)(n-1)}q^{(k-i)(n-1)^2}aq^{2n-2}\widetilde{\alpha}_{n-1}\right\rbrace.
\end{align}
The formula for $\beta_n^\fin$ depends on whether $i\leq k$ or not.
If $i\leq k$, we have:
\begin{align}
\beta_n^\fin=
\sum_{n\geq j_1\geq \cdots j_k\geq 0}
\frac{a^{j_1+\cdots +j_k} q^{j_1^2+\cdots +j_k^2-j_1-\cdots -j_i}}{(q)_{n-j_1}(q)_{j_1-j_2}\cdots (q)_{j_{k-1}-j_k}}\beta_{j_k}
\end{align}
If $i>k$, we have a total of $i+|k-i|=2i-k$ moves, and after some simplification, we have:
\begin{align}
\beta_n^\fin=
\sum_{n\geq j_1\geq \cdots j_{2i-k}\geq 0}
&(-1)^{j_{i}+j_{2i-k}}
\frac{a^{j_1+\cdots +j_{i-1} - (j_{i+1}+\cdots + j_{2i-k-1})}}{(q)_{n-j_1}(q)_{j_1-j_2}\cdots (q)_{j_{2i-k-1}-j_{2i-k}}}\beta_{j_{2i-k}}\times\nonumber\\
&\times  q^{j_1^2+\cdots +j_{i-1}^2 - (j_{i+1}^2+\cdots+j_{2i-k-1}^2) - (j_1+\cdots +j_i) + 
	\left(\binom{j_{i}-j_{i+1}}{2} + \cdots  + \binom{j_{2i -k-1}-j_{2i -k}}{2}\right)
}.
\end{align}
For $i=0$, asserting 
that $\alpha^\fin_n,\beta^\fin_n$ is indeed a Bailey pair with base $a$ and letting $n\rightarrow\infty$, we see:
\begin{align}
(q)_{\infty}\beta_\infty^\fin=
\frac{1}{(a)_{\infty}}\sum_{t\geq 0}a^{kt}q^{kt^2}(1-aq^{2t})\widetilde{\alpha}_t.
\label{eqn:lim1prodi=0}
\end{align}
For $i\geq 0$, following the same procedure with base $aq^{-1}$, we see:
\begin{align}
(q)_{\infty}\beta_\infty^\fin&=
\frac{1}{(a)_\infty}
\sum_{t\geq 0}
a^{it}q^{it^2-it}
\left\lbrace a^{(k-i)t}q^{(k-i)t^2}\widetilde{\alpha}_t-a^{(k-i)(t-1)}q^{(k-i)(t-1)^2}aq^{2t-2}\widetilde{\alpha}_{t-1}\right\rbrace\nonumber\\
&=\frac{1}{(a)_{\infty}}\sum_{t\geq 0}
a^{kt}q^{kt^2-it}
\left[1-a^{i+1}q^{2it+2t}\right]\widetilde{\alpha}_{t}.
\label{eqn:lim1prod}
\end{align}
Since setting $i=0$ in the right-hand side of \eqref{eqn:lim1prod} gets us 
the right-hand side of \eqref{eqn:lim1prodi=0}, we will simply use the right-hand side of \eqref{eqn:lim1prod} for all values of $i$.
	
\begin{rem}\label{rem:b1bc1}
	Whenever the step \eqref{eqn:b1} is followed by \eqref{eqn:bc1}, the formula for  $\beta$ simplifies. See  \eqref{eqn:b1bc1} below.
\end{rem}	
	
\subsection{Second forward move at the end}\label{sec:lim2}
Here, there are several cases.

\noindent\underline{If $i=0$}:
\begin{enumerate}
	\item Use the \eqref{eqn:f1} move $k-1$ times.
	\item Use the \eqref{eqn:f2} move once.
\end{enumerate}		
The final pair $\alpha^\fin_n, \beta^\fin_n$ can be seen to be:
\begin{align}
\alpha^\fin_n = \frac{(-q)_n}{(-a)_n}a^{kn}q^{(k-1)n^2 + (n^2-n)/2}\alpha_n,
\end{align}
\begin{align}
\beta_n^\fin = 
\sum_{n\geq j_1\geq \cdots j_k\geq 0}\frac{1}{(q)_{n-j_1}}
 \frac{(-q)_{j_1}a^{j_1+\cdots +j_k}q^{ \binom{j_1}{2} + j_{2}^2+\cdots +j_k^2}}{(-a)_{n}(q)_{j_1-j_{2}} \cdots (q)_{j_{k-1} - j_k} }\beta_{j_k}.
\end{align}
Asserting that $\alpha^\fin_n,\beta^\fin_n$ is indeed a Bailey pair with base $a$ and letting $n\rightarrow\infty$, we see:
\begin{align}
\beta^\fin_\infty=
\sum_{t\geq 0}
\frac{(-q)_t}{(-a)_t}a^{kt}q^{(k-1)t^2+(t^2-t)/2}\alpha_t \frac{1}{(q)_{\infty}(aq)_{\infty}}
\end{align}
Multiplying both sides by $(q)_{\infty}$ and simplifying, we see:
\begin{align}
(q)_{\infty}\beta^\fin_\infty=\frac{1}{(a)_{\infty}}\sum_{t\geq 0}\frac{(-q)_t}{(-a)_t}a^{kt}q^{(k-1)t^2+(t^2-t)/2}(1-aq^{2t})\widetilde{\alpha}_t. 
\label{eqn:F2endi=0}
\end{align}

\noindent\underline{If $i=1$}:
\begin{enumerate}
	\item Use \eqref{eqn:f1} move $k-1$ times.
	\item Use \eqref{eqn:bc2}.
\end{enumerate}
It can be seen that:
\begin{align}
\alpha_n^\fin = \frac{(-q)_na^nq^{(n^2-n)/2-n}}{(-a/q)_n}
\left\lbrace
{a^{(k-1)n}q^{(k-1)n^2}}\widetilde{\alpha}_n-
{a^{(k-1)(n-1)}q^{(k-1)(n-1)^2}aq^{2n-2}}\widetilde{\alpha}_{n-1}
\right\rbrace,
\end{align}
\begin{align}
\beta_n^\fin = 
\sum_{n\geq j_1\geq \cdots \geq j_k\geq 0}\frac{1}{(-a/q)_{n}(q)_{n-j_1}}
\frac{(-q)_{j_1}a^{j_1+\cdots +j_k}q^{\binom{j_1}{2} - j_1 + j_{2}^2+\cdots +j_k^2}}{(q)_{j_1-j_{2}} \cdots (q)_{j_{k-1} - j_k} }\beta_{j_k}.
\end{align}
Asserting that $\alpha^\fin_n,\beta^\fin_n$ is indeed a Bailey pair with base $aq^{-1}$, letting $n\rightarrow\infty$, multiplying by $(q)_{\infty}$, we see:
\begin{align}
(q)_{\infty}\beta^\fin_\infty
&=
\frac{\alpha_0}{(a)_{\infty}}+
	\sum_{t\geq 1}\frac{(-q)_t a^tq^{(t^2-t)/2-t}}{(-a/q)_t(a)_{\infty}}
\left\lbrace
{a^{(k-1)t}q^{(k-1)t^2}}\widetilde{\alpha}_t-
{a^{(k-1)(t-1)}q^{(k-1)(t-1)^2}aq^{2t-2}}\widetilde{\alpha}_{t-1}
\right\rbrace\nonumber\\
&=\frac{1}{(a)_{\infty}}\sum_{t\geq 0} a^{kt} q^{(k-1)t^2 +(t^2-t)/2-t }\frac{(-q)_t}{(-a/q)_t}
\left[1-\frac{1+q^{t+1}}{1+aq^{t-1}}a^2q^{3t+1}\right]\widetilde{\alpha}_t.
\label{eqn:F2endi=1}
\end{align}

\noindent\underline{If $i>1$}:
\begin{enumerate}
	\item If $k-i\geq 0$, use \eqref{eqn:f1} move $k-i$ times, else use \eqref{eqn:b1} $|k-i|$ times
	\item Use \eqref{eqn:bc1}
	\item Use \eqref{eqn:f1} $i-2$ times (note that the base is now $aq^{-1}$)
	\item Use the \eqref{eqn:f2} once (note that the base is $aq^{-1}$).
\end{enumerate}

After a bit of simplification, we see:
\begin{align}
\alpha_n^\fin 
=\frac{(-q)_n a^{in}q^{in^2-in-\frac{n^2+n}{2}}}{(-a/q)_n}
\left\lbrace  
a^{(k-i)n}q^{(k-i)n^2}\widetilde{\alpha}_n -
a^{(k-i)(n-1)}q^{(k-i)(n-1)^2}aq^{2n-2}\widetilde{\alpha}_{n-1}
\right\rbrace.
\end{align}
The formula for $\beta_n^\fin$ depends on whether $i\leq k$ or not.  If $i\leq k$,
\begin{align}
\beta_n^\fin
=
\sum_{n\geq j_1\geq \cdots \geq j_k\geq 0}\frac{1}{(-a/q)_{n}(q)_{n-j_1}}
\frac{(-q)_{j_1}a^{j_1+\cdots +j_k}q^{\binom{j_1}{2} + j_{2}^2+\cdots +j_k^2 - (j_1+\cdots+j_i)}}{(-a)_{j_1}(q)_{j_1-j_{2}} \cdots (q)_{j_{k-1} - j_k} }\beta_{j_k}.
\end{align}
If $i>k$, then there are a total of $i + |k-i| = 2i-k$ moves. After simplification, we have:
\begin{align}
\beta_n^\fin
=
\sum_{n\geq j_1\geq \cdots j_{2i-k}\geq 0}
&
(-1)^{j_{i}+j_{2i-k}}
\frac{ (-q)_{j_1} \,\,
	a^{j_1+\cdots +j_{i-1} - (j_{i+1}+\cdots + j_{2i-k-1})}}{(-a/q)_{n}(q)_{n-j_1}(q)_{j_1-j_2}\cdots (q)_{j_{2i-k-1}-j_{2i-k}}}\beta_{j_{2i-k}}\times
\nonumber\\
&\times  q^{\binom{j_1}{2}+j_2^2+\cdots +j_{i-1}^2 - (j_{i+1}^2+\cdots+j_{2i-k-1}^2) - (j_1+\cdots +j_i) + 
	\left(\binom{j_{i}-j_{i+1}}{2} + \cdots  + \binom{j_{2i -k-1}-j_{2i -k}}{2}\right)
}.
\end{align}

Asserting that $\alpha^\fin_n,\beta^\fin_n$ is indeed a Bailey pair with base $aq^{-1}$, letting $n\rightarrow\infty$, multiplying by $(q)_{\infty}$, we see:
\begin{align}
(q)_{\infty}\beta^\fin_\infty
&=
\frac{\alpha_0}{(a)_{\infty}}+
\sum_{t\geq 1}\frac{(-q)_t a^{it}q^{it^2-it-\frac{t^2+t}{2}}}{(-a/q)_t(a)_{\infty}}
\left\lbrace
{a^{(k-i)t}q^{(k-i)t^2}}\widetilde{\alpha}_t-
{a^{(k-i)(t-1)}q^{(k-i)(t-1)^2}aq^{2t-2}}\widetilde{\alpha}_{t-1}
\right\rbrace\nonumber\\
&=\frac{1}{(a)_{\infty}}\sum_{t\geq 0} a^{kt}q^{kt^2 - it - \frac{t^2+t}{2}}\frac{(-q)_t}{(-a/q)_t}\left[1-\frac{1+q^{t+1}}{1+aq^{t-1}}a^{i+1}q^{t-1+2it} \right]\widetilde{\alpha}_t.\label{eqn:F2endi>0}
\end{align}
\begin{rem}
For $a=q^2$ (which is the case when we will use this section), if we let $i=1$ in the RHS of \eqref{eqn:F2endi>0}, we get exactly the RHS of 
\eqref{eqn:F2endi=1}. If we take $i=0$ instead, we get the RHS of \eqref{eqn:F2endi=0} times $1/(1+q)$. 
Keeping these in mind, we shall use the RHS of \eqref{eqn:F2endi>0} for all values of $i$.
\end{rem}

\begin{rem}
	It is possible to use the moves given for $i>1$ for $i=0$ and $i=1$. One has to replace $i-2$ \eqref{eqn:f1} moves of step 3
	with $|i-2|$ \eqref{eqn:b1} moves. However, this makes the sum-sides have more summations.
\end{rem}

\subsection{Second forward move at the beginning}\label{sec:lim3}

\begin{enumerate}
	\item Use \eqref{eqn:f2} once.
	\item If $k-i-1\geq 0$ then use \eqref{eqn:f1} $k-i-1$ times, else use the \eqref{eqn:b1} $|k-i-1|$ times.
	\item If $i\geq 1$, use \eqref{eqn:bc1}.
	\item Noting that the base is $aq^{-1}$, if $i \geq 1$, use \eqref{eqn:f1} $i-1$ times.
\end{enumerate}		

We see that:
\begin{align}
\alpha_n^\fin=a^{in}q^{i(n^2-n)}
\left\lbrace
\frac{(-a)_na^{(k-i)n}q^{\binom{n}{2}+(k-i-1)n^2}}{(-q)_n}\widetilde{\alpha}_t-
\frac{(-a)_{n-1}a^{(k-i)(n-1)}q^{\binom{n-1}{2}+(k-i-1)(n-1)^2}aq^{2n-2}}{(-q)_{n-1}}\widetilde{\alpha}_{t-1}
\right\rbrace.
\end{align}
The formula for $\beta_n^\fin$ depends on the sign of $k-i-1$. If $k-i-1\geq 0$, we get:
\begin{align}
\beta_n^\fin=\sum_{n\geq j_1\geq \cdots \geq j_k\geq 0}
\frac{a^{j_1+\cdots + j_k} q^{j_1^2+\cdots + j_{k-1}^2 - (j_1+\cdots+ j_i) + \binom{j_k}{2}} (-q)_{j_k} }
{(q)_{n-j_1}(q)_{j_1-j_2}\cdots (q)_{j_{k-1}-j_k} (-a)_{j_{k-1}} }\beta_{j_k},
\label{eqn:lim3beta1}
\end{align}
where $(-a)_{j_{k-1}}$ is taken to be $(-a)_n$ if $k=1$.
If $k-i-1<0$, we have a total of $1-(k-i-1)+1+i-1=2i-k+2$ moves. We get:
\begin{align}
\beta_n^\fin=
&\sum_{n\geq j_1\geq \cdots \geq j_{2i-k+2}\geq 0}
(-1)^{j_i + j_{2i-k+1}}
\frac{(-q)_{j_{2i-k+2}}\,\, a^{j_1+\cdots + j_{i-1} -(j_{i+1}+\cdots+j_{2i-k}) + j_{2i-k+2} } }
{(q)_{n-j_1}(q)_{j_1-j_2}\cdots (q)_{j_{2i-k+1}-j_{2i-k+2}} (-a)_{j_{2i-k+1}} }\beta_{j_{2i-k+2}}\times \nonumber\\
&\times q^{j_1^2+\cdots+j_{i-1}^2-(j_{i+1}^2+\cdots+j_{2i-k}^2) - (j_1+\cdots +j_i)
+\left(\binom{j_i-j_{i+1}}{2}+\cdots+\binom{j_{2i-k} - j_{2i-k+1}}{2} \right) + \binom{j_{2i-k+2}}{2}}.
\label{eqn:lim3beta2}
\end{align}
Asserting that $\alpha^\fin_n,\beta^\fin_n$ is indeed a Bailey pair with base $aq^{-1}$, letting $n\rightarrow\infty$, multiplying by $(q)_{\infty}$, we see:
\begin{align}
&(q)_{\infty}\beta^\fin_\infty
=\frac{\alpha_0}{(a)_{\infty}}\nonumber\\
&+\sum_{t\geq 0}
\frac{a^{it}q^{i(t^2-t)}}{(a)_{\infty}}
\left\lbrace
\frac{(-a)_t}{(-q)_t}a^{(k-i)t}q^{\binom{t}{2}+(k-i-1)t^2}\widetilde{\alpha}_t-
\frac{(-a)_{t-1}}{(-q)_{t-1}}a^{(k-i)(t-1)}q^{\binom{t-1}{2}+(k-i-1)(t-1)^2}aq^{2t-2}\widetilde{\alpha}_{t-1}
\right\rbrace\nonumber\\
&=\frac{1}{(a)_\infty}
\sum_{t\geq 0}
a^{kt}q^{kt^2 -it-\frac{t^2+t}{2}}\frac{(-q)_t}{(-a)_t}\left[1-a^{i+1}q^{2t(i+1)}\right]\widetilde{\alpha}_t.
\label{eqn:F2begprod}
\end{align}

\begin{rem}\label{rem:f2b1}
	Whenever the step \eqref{eqn:f2} is followed by \eqref{eqn:b1}, the formula for $\beta$ simplifies. See  \eqref{eqn:f2b1} below.
\end{rem}

\section{A collection of Bailey pairs}
Recalling \eqref{eqn:alphatilde}, we have the following Bailey pairs.
The relation of these Bailey pairs with the moduli will be explained below.
{\renewcommand{\arraystretch}{2.5}
	\begin{longtable}{|c||c||c|c c c||c||c|}
		\caption{Bailey Pairs}
		\label{table:pairs}
		\endfirsthead
		\endhead
		\hline
		\# & $a$ & $\beta_n$ & $\widetilde{\alpha}_m=\widetilde{\alpha}_{3n-1}$ & $\widetilde{\alpha}_m=\widetilde{\alpha}_{3n}$ & 
		$\widetilde{\alpha}_m=\widetilde{\alpha}_{3n+1}$ & {c.f.}
		& moduli \\
		\hline\hline
		1 & $q$ & $\dfrac{1}{(q)_{2n}}$ & 
		$-q^{\frac{1}{3}(2m^2-m)}$
		&$q^{\frac{1}{3}(2m^2-m)}$
		&$0$ &  A1, \cite{Sla-pairs}
		& $12k+8$, $12k+2$\\
		\hline
		 2 & $q^2$ & $\dfrac{1}{(q^2;q)_{2n}}$ & 
		0
		&$q^{\frac{1}{3}(2m^2+m)}$
		&$-q^{\frac{1}{3}(2m^2+m)}$
		&  A2, \cite{Sla-pairs}
		& $12k+8$, $12k+2$\\
		\hline
		3 & $q$ & $\dfrac{q^{n^2-n}}{(q)_{2n}}$
		&$-q^{\frac{1}{3}(m^2-2m)}$
		&$q^{\frac{1}{3}(m^2-2m)}$
		&$0$ &  A7, \cite{Sla-pairs}
		& $12k+4$, $12k-2$\\
		\hline
		4 & $q^2$ & $\dfrac{q^{n^2}}{(q^2;q)_{2n}}$ & 
		$0$
		&$q^{\frac{1}{3}(m^2-m)}$
		&$-q^{\frac{1}{3}(m^2-m)}$
		&  A6, \cite{Sla-pairs}
		& $12k+4$, $12k-2$\\
		\hline
		5 & $q$ & $\dfrac{(-1;q^3)_n}{(q)_{2n}(-1;q)_n}$ & 
		$-q^{\frac{1}{2}(m^2-m)}$
		&$q^{\frac{1}{2}(m^2-m)}$
		&$0$
		&  P1, \cite{McLSil-1824}
		& $12k+6$, $12k$\\
		\hline
	\end{longtable}
	\renewcommand{\arraystretch}{1}
}

To obtain each of these Bailey pairs, we first start with the pair $\alpha_n',\beta_n$ given in the penultimate column and then shift the base as in Lemma \ref{lem:baseshift}.
As an example, consider the pair A1 from \cite{Sla-pairs} (base $a=1$):
\begin{align}
\alpha_m'=
\begin{cases}
1 & m=0\\
-q^{6n^2-5n+1} & m=3n-1\\
q^{6n^2-n} +  q^{6n^2+n} & m=3n, n\neq 0\\
-q^{6n^2+5n+1} & m=3n+1.
\end{cases},
\quad\quad\beta_n=\frac{1}{(q)_{2n}}.
\end{align}
Now we prove by induction that our intended Bailey pair given in the table above satisfies the recurrence \eqref{eqn:baseshiftalphatilde_rec}.
Note that the original base $a$ is $1$. We have $\widetilde{\alpha}_0=\alpha'_0=1$.
Then, we have:
\begin{align}
&q^{2(3n-2)+1}\widetilde{\alpha}_{3n-2}+\alpha'_{3n-1} = 
0-q^{6n^2-5n+1}  = -q^{6n^2-5n+1} = -q^{\frac{1}{3}(2m^2-m)}, \,\,(m=3n-1)\label{eqn:A13n-1}\\
&q^{2(3n-1)+1}\widetilde{\alpha}_{3n-1}+\alpha'_{3n} = 
q^{6n-1}(-q^{6n^2-5n+1})+q^{6n^2-n}+q^{6n^2+n}  = q^{6n^2-n}=q^{\frac{1}{3}(2m^2-m)}, \,\,(m=3n)\label{eqn:A13n}\\
&q^{2(3n)+1}\widetilde{\alpha}_{3n}+\alpha'_{3n+1} = 
q^{6n+1}q^{6n^2-n}-q^{6n^2+5n+1}  = 0.\label{eqn:A13n+1}
\end{align}
As required, right-hand sides of  \eqref{eqn:A13n-1}, \eqref{eqn:A13n}, \eqref{eqn:A13n+1} are exactly the values of 
$\widetilde{\alpha}_{3n-1}$, $\widetilde{\alpha}_{3n}$, $\widetilde{\alpha}_{3n+1}$, respectively. 
The other pairs can be handled similarly.

\section{The principal characters of $A_2^{(2)}$ modules}
For this section, the main references are \cite{LepMil-a22, Sil-book, Bos-char, Kac-book}.
Pick a level $\ell\in\ZZ_{>0}$. There are $1+\left[\frac{\ell}{2}\right]$ many distinct irreducible standard (i.e., integrable highest-weight) $A_2^{(2)}$ modules up to isomorphisms at level $\ell$,
which we enumerate as $L(s_0,s_1)$ where $s_0,s_1\in\ZZ_{\geq 0}, s_0+2s_1=\ell$.
We decompose these modules with respect to the principal Heisenberg subalgebra of $A_2^{(2)}$, and consider the space of highest weight vectors. We call these spaces vacuum spaces and denote them by $\Omega(s_0,s_1)$. We then consider the principally specialized characters of these vacuum spaces, and denote them by $\chi(\Omega(s_0,s_1))$. We call $\chi(\Omega(s_0,s_1))$ the \emph{principal character} of $L(s_0,s_1)$.
We have the following formula:
\begin{align}
\chi(\Omega(s_0,s_1))
= \frac{(q^{s_1+1},q^{s_0+s_1+2},q^{s_0+2s_1+3};q^{s_0+2s_1+3})_{\infty}
(q^{s_0+1},q^{s_0+4s_1+5};q^{2(s_0+2s_1+3)})_{\infty}}{(q)_{\infty}}.
\end{align}
Observe that this character is a product that is ``periodic'' modulo $2(s_0+2s_1+3)=2\ell+6$. We thus say that
the modulus of $\chi(\Omega(s_0,s_1))$ is $2\ell+6$.
Importantly, this character can be obtained by making the following substitution in QTPI \eqref{eqn:QTPI} and dividing by $(q)_{\infty}$: 
\begin{align}
\chi(\Omega(s_0,s_1))
= \frac{Q(q^{s_0+2s_1+3},q^{-s_1-1})}{(q)_{\infty}}
= \frac{Q(q^{\ell+3},q^{-s_1-1})}{(q)_{\infty}}.
\label{eqn:qtpisubs}
\end{align}

With $k\geq 1$, the Bailey pairs mentioned above give rise to the following characters.
\renewcommand{\arraystretch}{1.5}
\begin{longtable}{|c||c||c|c|c||c||c|}
\caption{Modules and Identities}
\label{table:ids}
\endfirsthead
\endhead
\hline
Pair \# & Moves	& Level & \# of Modules & Modulus & $i$ & $s_1$ \\
\hline\hline
1  & \ref{sec:lim1} & $6k+1$ & $3k+1$ & $12k+8$ & $0\leq i\leq 3k$ & $s_1=i$\\
1  & \ref{sec:lim3} & $6k-2$ & $3k$ & $12k+2$ & $0\leq i\leq 3k-1$ & $s_1=i $ \\
\hline
2  & \ref{sec:lim1} & $6k+1$ & $3k+1$ & $12k+8$ & $0\leq i\leq 3k$ & $s_1=3k-i$\\
2  & \ref{sec:lim2} & $6k-2$ & $3k$ & $12k+2$ & $0\leq i\leq 3k-1$ & $s_1=3k-1-i$\\
\hline
3  & \ref{sec:lim1} & $6k-1$ & $3k$ & $12k+4$ & $0\leq i\leq 3k-1$ & $s_1=i$\\
3  & \ref{sec:lim3} & $6k-4$ & $3k-1$ & $12k-2$ & $0\leq i\leq 3k-2$ & $s_1=i$\\
\hline
4  & \ref{sec:lim1} & $6k-1$ & $3k$ & $12k+4$ & $0\leq i\leq 3k-1$ & $s_1=3k-1-i$\\
4  & \ref{sec:lim2} & $6k-4$ & $3k-1$ & $12k-2$ & $0\leq i\leq 3k-2$ & $s_1=3k-2-i$\\
\hline
5  & \ref{sec:lim1} & $6k$ & $3k+1$ & $12k+6$ & $0\leq i\leq 3k$ & $s_1=i$\\
5 & \ref{sec:lim3} & $6k-3$ & $3k-1$ & $12k$ & $0\leq i\leq 3k-2$ & $s_1=i$\\
\hline	
\end{longtable}
\renewcommand{\arraystretch}{1}

We now show that we indeed get the promised characters via our Bailey pairs. We shall ignore the $1/(a)_{\infty}$ factor from 
\eqref{eqn:lim1prod}, \eqref{eqn:F2endi=0}, \eqref{eqn:F2endi=1}, \eqref{eqn:F2endi>0}, \eqref{eqn:F2begprod}, etc.

Consider the first pair (base is $a=q$) with moves in Section \ref{sec:lim1}.
Equation \eqref{eqn:lim1prod} gives:
\begin{align}
&\sum_{t\geq 0}(q^{kt^2+kt-it}-q^{kt^2+t(i+2)+kt+i+1})\widetilde{\alpha}_t
=\left(\sum_{
	\substack{t=3n\\n\geq 0 }}
+\sum_{
	\substack{t=3n-1\\n\geq 1 }}\right)(q^{kt^2+kt-it}-q^{kt^2+t(i+2)+kt+i+1})\widetilde{\alpha}_t
\nonumber\\
&=\sum_{\substack{n\geq 0 }}
(q^{9kn^2-3in+3kn}-q^{9kn^2+3in+3kn+i+6n+1})
+
\sum_{\substack{n\geq 1}}
(-q^{9kn^2-3in-3kn+i}+q^{9kn^2+3in-3kn+6n-1})
\nonumber\\
&=Q(q^{6k+4},q^{-i-1})\quad\quad\dots (\mathrm{by}\,\,\eqref{eqn:QTPI-II}),
\end{align}
as required by \eqref{eqn:qtpisubs}. Similar calculations work for the third and fifth pair while using Section \ref{sec:lim1}.

Consider the second pair (base is $a=q^2$) with moves in Section \ref{sec:lim1}.
The product side is:
\begin{align}
&\sum_{t\geq 0}(q^{2kt+kt^2-it}-q^{2(kt+i+1)+kt^2+t(i+2)})\widetilde{\alpha}_t
=\left(\sum_{
	\substack{t=3n\\n\geq 0 }}
+\sum_{
	\substack{t=3n+1\\n\geq 0 }}\right)(q^{2kt+kt^2-it}-q^{2(kt+i+1)+kt^2+t(i+2)})\widetilde{\alpha}_t
\nonumber\\
&=
\sum_{
	\substack{n\geq 0 }}
(q^{9kn^2-3in+6kn+6n^2+n}-q^{9kn^2+3in+6kn+6n^2+2i+7n+2})\nonumber\\
&\quad\quad
+
\sum_{
	\substack{n\geq 0 }}
(-q^{9kn^2-3in+12kn+6n^2-i+3k+5n+1}+q^{9kn^2+3in+12kn+6n^2+3i+3k+11n+5})
\nonumber\\
&=Q(q^{6k+4},q^{-3k+i-1})\quad\quad\dots (\mathrm{by}\,\,\eqref{eqn:QTPI-III}),
\end{align}
as required by \eqref{eqn:qtpisubs}. Similar calculations work for the fourth pair while using Section \ref{sec:lim1}.

Consider first pair with moves in Section \ref{sec:lim3}. The base is $a=q$. We have, from \eqref{eqn:F2begprod} with $a\mapsto q$:
\begin{align}
&\sum_{t\geq 0}
q^{kt+ kt^2 -it-\frac{t^2+t}{2}}\left[1-q^{(2t+1)(i+1)}\right]\widetilde{\alpha}_t\nonumber\\
&=\left(\sum_{\substack{t=3n,\\ n\geq 0}}+\sum_{\substack{t=3n-1,\\ n\geq 1}}\right)
\left(q^{kt+ kt^2 -it-\frac{t^2+t}{2}} -q^{kt+ kt^2 -it-\frac{t^2+t}{2}+(2t+1)(i+1)}\right)\widetilde{\alpha}_t
\\
&=\sum_{n\geq 0} (q^{3nk+9kn^2-3ni+(3/2)n^2-(5/2)n}-q^{3nk+9kn^2+3ni+(3/2)n^2+(7/2)n+i+1})\nonumber\\
&\quad\quad + \sum_{n\geq 1} (-q^{-3nk+9kn^2-3ni+i+(3/2)n^2-(7/2)n+1} + q^{-3nk+9kn^2+3ni+(3/2)n^2+(5/2)n}) \nonumber\\
&=Q(q^{6k+1},q^{-i-1})\quad\quad\dots(\mathrm{by}\,\,\eqref{eqn:QTPI-II}),
\end{align}
as required. Similar calculations are used for third and fifth pair with Section \ref{sec:lim3}.

Consider the second pair (base is $a=q^2$) with moves in Section \ref{sec:lim2}.
Equation \eqref{eqn:F2endi>0}, after some simplification becomes:
\begin{align}
&\sum_{t\geq 0} q^{2kt + kt^2 - it - \frac{t^2+t}{2}}(1-q^{2i+2+t-1+2it})\widetilde{\alpha}_t
= \left(\sum_{
	\substack{t=3n\\n\geq 0 }}
+\sum_{
	\substack{t=3n+1\\n\geq 0 }}\right)
q^{2kt + kt^2 - it - \frac{t^2+t}{2}}(1-q^{2i+2+t-1+2it})\widetilde{\alpha}_t\nonumber\\
&=\sum_{n\geq 0}(q^{6nk+9kn^2-3ni+(3/2)n^2-(1/2)n}-q^{6nk+9kn^2+3ni+(3/2)n^2+(5/2)n+2i+1})\nonumber\\
&\quad\quad + \sum_{n\geq 0}(q^{12nk+3k+9kn^2-3ni-i+(3/2)n^2+(1/2)n}-q^{12nk+3k+9kn^2+3ni+3i+(3/2)n^2+(7/2)n+2})\nonumber\\
&=Q(q^{6k+1},q^{-3k+i})\quad\quad\dots(\mathrm{by}\,\,\eqref{eqn:QTPI-III}).
\end{align}
Similar calculation works for the fourth pair used in conjunction with Section \ref{sec:lim2}.

\section{Examples}
In this section, we demonstrate examples for levels $2,\dots, 7$. We leave some exponents un-simplified to better explain how they have been obtained.
Some sides could be simplified using the following lemma; recall Remarks  \ref{rem:b1bc1}, \ref{rem:f2b1}:
\begin{lem}
	For fixed $j_1\geq j_3\geq 0$, we have:
	\begin{align}
\sum_{j2:\,\, j_1\geq j_2\geq j_3}(-1)^{j_2+j_3}\frac{q^{-j_2}}{(q)_{j_1-j_2}} \frac{q^{\binom{j_2-j_3}{2}}}{(q)_{j_2-j_3}}
&=\frac{1}{q^{j_1}}\cdot\begin{cases}-1 & j_1=j_3+1\\
1 & j_1=j_3\\
0 & \mathrm{otherwise}
\end{cases}
\label{eqn:b1bc1}.\\
\sum_{j2:\,\, j_1\geq j_2\geq j_3}(-1)^{j_2}\frac{q^{\binom{j_1-j_2}{2}}}{(q)_{j_1-j_2}(q)_{j_2-j_3}(-a)_{j_2}}
&=(-1)^{j_3}\frac{a^{j_1-j_3}q^{\binom{j_1-j_3}{2} + \binom{j_1}{2} - \binom{j_3}{2}} }{(-a)_{j_1}(q)_{j_1-j_3}}
\label{eqn:f2b1}.
\end{align}
\end{lem}
\begin{proof}
Equation \eqref{eqn:b1bc1} is immediate from the $q$-binomial theorem, \cite[Eq.\ 3.3.6]{And-book}.

For \eqref{eqn:f2b1}, we prove that for all $N,t\geq 0$,
\begin{align}
\sum_{i=0}^N(-1)^i\frac{q^{\binom{i}{2}}(q)_{N}}{(q)_i(q)_{N-i}(-a)_{N-i+t}}=(-1)^N\frac{a^Nq^{N(N+t-1)}}{(-a)_{N+t}}.
\label{eqn:f2b1other}
\end{align}
Both sides are solutions to the recurrence:
\begin{align}
f(N+1,t) = f(N,t+1)-q^Nf(N,t),\quad f(0,t)=\frac{1}{(-a)_t}.
\end{align}
Indeed, the recurrence for the left-hand side of \eqref{eqn:f2b1other} could be deduced easily from the properties of Gaussian polynomials \cite[Eqn.\ 3.3.3]{And-book},
and the recurrence for the right-hand side is straightforward.
Now \eqref{eqn:f2b1} follows from \eqref{eqn:f2b1other} by taking $i=j_1-j_2$, $t=j_3$, $N=j_1-j_3$. 
\end{proof}

\subsection{Level $2$}

Using the third pair and Section \ref{sec:lim3}, with $a=q$, $k=1$, $i=0,1$, we get:
\begin{align}
&\frac{1}{(-q)_{\infty}}\sum_{j_1\geq 0} \frac{q^{j_1}q^{\binom{j_1}{2}} (-q)_{j_1}q^{j_1^2-j_1}}{(q)_{2j_1}}
= \frac{(q,q^4,q^5;\,\,q^5)_{\infty}(q^3,q^7;\,\,q^{10})_\infty}{(q)_{\infty}},\\
&\sum_{j_1\geq j_2\geq j_3\geq 0}
(-1)^{j_1+j_2}\frac{(-q)_{j_3}q^{j_3}q^{-j_1+\binom{j_1-j_2}{2} +\binom{j_3}{2}}}{(q)_{j_1-j_2}(q)_{j_2-j_3}(-q)_{j_2}}
\frac{q^{j_3^2-j_3}}{(q)_{2j_3}}
=
\frac{(q^2,q^3,q^5;\,\,q^5)_{\infty}(q^1,q^{9};\,\,q^{10})_\infty}{(q)_{\infty}}.
\end{align}
The triple-sum here simplifies, using \eqref{eqn:f2b1}:
\begin{align}
&\sum_{j_1\geq j_3\geq 0}
(-1)^{j_1+j_3}\frac{q^{\binom{j_1}{2}+\binom{j_1-j_3}{2}}(-q)_{j_3}}{(q)_{j_1-j_3}(-q)_{j_1}}
\frac{q^{j_3^2-j_3}}{(q)_{2j_3}}
=
\frac{(q^2,q^3,q^5;\,\,q^5)_{\infty}(q^1,q^{9};\,\,q^{10})_\infty}{(q)_{\infty}}.
\end{align}

If we instead use the fourth pair with Section \ref{sec:lim2},
with $a=q^2$, $k=1$, $i=0,1$:
\begin{align}
\sum_{j_1\geq 0}\frac{(-q)_{j_1}q^{2j_1+\binom{j_1}{2}}}{(-q^2)_{\infty}}\frac{q^{j_1^2}}{(q^2;q)_{2j_1}}
=(1-q^2)\frac{(q^2,q^3,q^5;\,\,q^5)_{\infty}(q^1,q^{9};\,\,q^{10})_\infty}{(q)_{\infty}},\\
\sum_{j_1\geq 0}\frac{(-q)_{j_1} q^{j_1+\binom{j_1}{2}}}{(-q)_{\infty}} \frac{q^{j_1^2}}{(q^2;q)_{2j_1}}
=(1-q)\frac{(q^1,q^4,q^5;\,\,q^5)_{\infty}(q^3,q^{7};\,\,q^{10})_\infty}{(q^2;q)_{\infty}}.
\end{align}

\subsection{Level $3$}
At this level, we get $q$-series for Capparelli's identities \cite{Cap-1}.
We use the fifth Bailey pair and Section \ref{sec:lim3}.
With $a=q$, $k=1$ and $i=0,1$ we get:
\begin{align}
&\frac{1}{(-q)_{\infty}}\sum_{j_1\geq 0} \frac{q^{j_1}q^{\binom{j_1}{2}} (-q)_{j_1}(-1;q^3)_{j_1}}{(q)_{2j_1}(-1;q)_{j_1}}
= \frac{1}{(-q)_{\infty}}
\left( 1 + 
\sum_{j_1\geq 1} 
\frac{q^{j_1}q^{\binom{j_1}{2}} (1+q^{j_1})(-q^3;q^3)_{j_1-1}}{(q;q)_{2j_1}}
\right)\nonumber\\
&\quad\quad\quad= \frac{(q,q^5,q^6;\,\,q^6)_{\infty}(q^4,q^8;\,\,q^{12})_\infty}{(q)_{\infty}},\\
&\sum_{j_1\geq j_2\geq j_3\geq 0}
(-1)^{j_1+j_2}\frac{(-q)_{j_3}q^{j_3}q^{-j_1+\binom{j_1-j_2}{2} +\binom{j_3}{2}}}{(q)_{j_1-j_2}(q)_{j_2-j_3}(-q)_{j_2}}
\frac{(-1;q^3)_{j_3}}{(q;q)_{2j_3}(-1;q)_{j_3}}
=
\frac{(q^2,q^4,q^6;\,\,q^6)_{\infty}(q^2,q^{10};\,\,q^{12})_\infty}{(q)_{\infty}}.
\end{align}
The triple-sum here simplifies, using \eqref{eqn:f2b1}:
\begin{align}
&\sum_{j_1\geq j_3\geq 0}
(-1)^{j_1+j_3}\frac{q^{\binom{j_1}{2}+\binom{j_1-j_3}{2}}(-q)_{j_3}}{(q)_{j_1-j_3}(-q)_{j_1}}
\frac{(-1;q^3)_{j_3}}{(q)_{2j_3}(-1;q)_{j_3}}
=
\frac{(q^2,q^4,q^6;\,\,q^6)_{\infty}(q^2,q^{10};\,\,q^{12})_\infty}{(q)_{\infty}}.
\end{align}

\subsection{Level 4} At this level, we get $q$-series related to Nandi's identities, \cite{Nan-thesis}, \cite{TakTsu-nandi}. Compare these identities with \cite{Sla-ids}.
We use the first Bailey pair and Section \ref{sec:lim3}.
With $a=q$, $k=1$ and $i=0,1,2$ we get:
\begin{align}
&\frac{1}{(-q)_{\infty}}\sum_{j_1\geq 0} \frac{q^{j_1}q^{\binom{j_1}{2}} (-q)_{j_1}}{(q)_{2j_1}}
= \frac{(q,q^6,q^7;\,\,q^7)_{\infty}(q^5,q^9;\,\,q^{14})_\infty}{(q)_{\infty}},\\
&\sum_{j_1\geq j_2\geq j_3\geq 0}
(-1)^{j_1+j_2}\frac{(-q)_{j_3}q^{j_3}q^{-j_1+\binom{j_1-j_2}{2} +\binom{j_3}{2}}}{(q)_{j_1-j_2}(q)_{j_2-j_3}(-q)_{j_2}}
\frac{1}{(q;q)_{2j_3}}
=
\frac{(q^2,q^5,q^7;\,\,q^7)_{\infty}(q^3,q^{11};\,\,q^{14})_\infty}{(q)_{\infty}},\\
&\sum_{j_1\geq\cdots\geq j_5\geq 0}
(-1)^{j_2+j_4}\frac{(-q)_{j_5} q^{j_1-j_3+j_5}q^{j_1^2-j_3^2-(j_1+j_2) + \binom{j_2-j_3}{2}+ \binom{j_3-j_4}{2} +  \binom{j_5}{2} }}
{(q)_{j_1-j_2}(q)_{j_2-j_3}(q)_{j_3-j_4}(q)_{j_4-j_5}(-q)_{j_4}}
\frac{1}{(q;q)_{2j_5}}\nonumber\\
&\quad\quad\quad=
\frac{(q^3,q^4,q^7;\,\,q^7)_{\infty}(q^1,q^{13};\,\,q^{14})_\infty}{(q)_{\infty}}.
\end{align}
The triple-sum simplifies, using \eqref{eqn:f2b1}:
\begin{align}
&
\sum_{j_1\geq j_3\geq 0}
(-1)^{j_1+j_3}\frac{q^{\binom{j_1}{2}+\binom{j_1-j_3}{2}}(-q)_{j_3}}{(q)_{j_1-j_3}(-q)_{j_1}}
\frac{1}{(q)_{2j_3}}
=
\frac{(q^2,q^5,q^7;\,\,q^7)_{\infty}(q^3,q^{11};\,\,q^{14})_\infty}{(q)_{\infty}}.
\end{align}
The quintuple-sum simplifies. The sum over $j_4$ is handled by \eqref{eqn:f2b1}, and then the sum over $j_2$ by \eqref{eqn:b1bc1}.
\begin{align}
&\sum_{j_1\geq j_2\geq j_3\geq j_5\geq 0}
(-1)^{j_2+j_5}\frac{(-q)_{j_5}q^{j_1^2-j_3^2-j_2+\binom{j_2-j_3}{2}+\binom{j_3-j_5}{2}+\binom{j_3}{2} }}{(q)_{j_1-j_2}(q)_{j_2-j_3}(q)_{j_3-j_5}(-q)_{j_3}}
\frac{1}{(q)_{2j_5}}
\nonumber\\
&
=
\sum_{j_3\geq j_5\geq 0}
(-1)^{j_3+j_5}(-q^{j_3}+q^{-j_3})\frac{(-q)_{j_5}q^{\binom{j_3-j_5}{2} +\binom{j_3}{2}}}{(q)_{j_3-j_5}(-q)_{j_3}}\frac{1}{(q)_{2j_5}}
\nonumber\\
&=
\frac{(q^2,q^5,q^7;\,\,q^7)_{\infty}(q^3,q^{11};\,\,q^{14})_\infty}{(q)_{\infty}}.
\end{align}

If we use the second Bailey pair and Section \ref{sec:lim2} and $a=q^2$, $k=1$ and $i=0,1,2$ we get:
\begin{align}
&\sum_{j_1\geq 0}
\frac{(-q)_{j_1}q^{2j_1 + \binom{j_1}{2}} }{(-q^2;q)_\infty}\frac{1}{(q^2;q)_{2j_1}} =
(1-q^2)\frac{(q^3,q^4,q^7;\,\,q^7)_{\infty}(q^1,q^{13};\,\,q^{14})_\infty}{(q)_{\infty}},\\
&\sum_{j_1\geq 0}\frac{(-q)_{j_1}q^{j_1+\binom{j_1}{2}}}{(-q)_{\infty}}\frac{1}{(q^2;q)_{2j_1}}
=
(1-q)\frac{(q^2,q^5,q^7;\,\,q^7)_{\infty}(q^3,q^{11};\,\,q^{14})_\infty}{(q)_{\infty}},\\
&\sum_{j_1\geq j_2\geq j_3\geq 0}
(-1)^{j_2+j_3}
\frac{(-q)_{j_1}q^{2j_1 + \binom{j_1}{2} -(j_1+j_2)+\binom{j_2-j_3}{2}}}{(-q)_{\infty}(q)_{j_1-j_2}(q)_{j_2-j_3}}
\frac{1}{(q^2;q)_{2j_3}}
=(1-q)\frac{(q,q^6,q^7;\,\,q^7)_{\infty}(q^5,q^9;\,\,q^{14})_\infty}{(q)_{\infty}}.
\end{align}
The triple-sum simplifies, using \eqref{eqn:b1bc1}:
\begin{align}
\sum_{j_3\geq 0}
\frac{(-q)_{j_3}q^{\binom{j_3}{2}}}{(-q)_{\infty}}[1-q^{j_3}-q^{2j_3+1}]
\frac{1}{(q^2;q)_{2j_3}}
=(1-q)\frac{(q,q^6,q^7;\,\,q^7)_{\infty}(q^5,q^9;\,\,q^{14})_\infty}{(q)_{\infty}}.
\end{align}

\subsection{Level 5}
Compare the identities here with \cite{Sla-ids}, \cite{TakTsu-A22A132} and \cite{AndEkeHel-vir34}.
We use the third pair with base $a=q$, moves in Section \ref{sec:lim1} and $k=1$, $i=0,1,2$:
\begin{align}
&\sum_{j_1\geq 0}{q^{j_1+j_1^2}}\frac{q^{j_1^2-j_1}}{(q)_{2j_1}}=
\frac{(q,q^7,q^8;\,\,q^8)_{\infty}(q^6,q^{10};\,\,q^{16})_\infty}{(q)_{\infty}},\\
&\sum_{j_1\geq 0}{q^{j_1^2}}\frac{q^{j_1^2-j_1}}{(q)_{2j_1}}=
\frac{(q^2,q^6,q^8;\,\,q^8)_{\infty}(q^4,q^{12};\,\,q^{16})_\infty}{(q)_{\infty}},\\
&\sum_{j_1\geq j_2\geq j_3\geq 0}
(-1)^{j_2+j_3}\frac{q^{j_1+j_1^2-(j_1+j_2)+\binom{j_2-j_3}{2}}}{(q)_{j_1-j_2}(q)_{j_2-j_3}}
\frac{q^{j_3^2-j_3}}{(q)_{2j_3}}
=
\frac{(q^3,q^5,q^8;\,\,q^8)_{\infty}(q^2,q^{10};\,\,q^{16})_\infty}{(q)_{\infty}}.
\end{align}
The triple-sum simplifies using \eqref{eqn:b1bc1}:
\begin{align}
\sum_{j_3\geq 0}
q^{j_3^2-j_3}(1-q^{2j_3})
\frac{q^{j_3^2-j_3}}{(q)_{2j_3}}
=
\sum_{j_3\geq 0}
\frac{q^{2(j_3^2+j_3)}}{(q)_{2j_3+1}}
=
\frac{(q^3,q^5,q^8;\,\,q^8)_{\infty}(q^2,q^{10};\,\,q^{16})_\infty}{(q)_{\infty}}.
\end{align}

If instead we use the fourth pair with base $a=q^2$, Section \ref{sec:lim1} and $k=1$, $i=0,1,2$, we get:
\begin{align}
&\sum_{j_1\geq 0}{q^{2j_1+j_1^2}}
\frac{q^{j_1^2}}{(q^2;q)_{2j_1}}=
(1-q)\frac{(q^3,q^5,q^8;\,\,q^8)_{\infty}(q^2,q^{10};\,\,q^{16})_\infty}{(q)_{\infty}},\\
&\sum_{j_1\geq 0}{q^{j_1+j_1^2}}
\frac{q^{j_1^2}}{(q^2;q)_{2j_1}}=
(1-q)\frac{(q^2,q^6,q^8;\,\,q^8)_{\infty}(q^4,q^{12};\,\,q^{16})_\infty}{(q)_{\infty}},\\
&\sum_{j_1\geq j_2\geq j_3\geq 0}
(-1)^{j_2+j_3}\frac{q^{2j_1+j_1^2-(j_1+j_2)+\binom{j_2-j_3}{2}}}{(q)_{j_1-j_2}(q)_{j_2-j_3}}
\frac{q^{j_3^2}}{(q^2;q)_{2j_3}}
=
(1-q)\frac{(q,q^7,q^8;\,\,q^8)_{\infty}(q^6,q^{10};\,\,q^{16})_\infty}{(q)_{\infty}}.
\end{align}
The triple-sum simplifies, using \eqref{eqn:b1bc1}:
\begin{align}
\sum_{j_3\geq 0}
q^{j_3^2}(1-q^{2j_3+1})
\frac{q^{j_3^2}}{(q^2;q)_{2j_3}}
= (1-q) + \sum_{j_3\geq 1}
\frac{q^{2j_3^2}}{(q^2;q)_{2j_3-1}}
=
(1-q)\frac{(q,q^7,q^8;\,\,q^8)_{\infty}(q^6,q^{10};\,\,q^{16})_\infty}{(q)_{\infty}}.
\end{align}

\subsection{Level 6} Compare the identities here with \cite{McLSil-1824}.
We use the fifth pair with base $a=q$, moves in Section \ref{sec:lim1} and $k=1$, $i=0,1,2,3$:
\begin{align}
&\sum_{j_1\geq 0}{q^{j_1+j_1^2}}
\frac{(-1;q^3)_{j_1}}{(q)_{2j_1}(-1;q)_{j_1}}=
\frac{(q,q^8,q^9;\,\,q^9)_{\infty}(q^7,q^{11};\,\,q^{18})_\infty}{(q)_{\infty}},\\
&\sum_{j_1\geq 0}{q^{j_1^2}}
\frac{(-1;q^3)_{j_1}}{(q)_{2j_1}(-1;q)_{j_1}}=
\frac{(q^2,q^7,q^9;\,\,q^9)_{\infty}(q^5,q^{13};\,\,q^{18})_\infty}{(q)_{\infty}},\\
&\sum_{j_1\geq j_2\geq j_3\geq 0}
(-1)^{j_2+j_3}\frac{q^{j_1+j_1^2-(j_1+j_2)+\binom{j_2-j_3}{2}}}{(q)_{j_1-j_2}(q)_{j_2-j_3}}
\frac{(-1;q^3)_{j_3}}{(q)_{2j_3}(-1;q)_{j_3}}
=
\frac{(q^3,q^6,q^9;\,\,q^9)_{\infty}(q^3,q^{15};\,\,q^{18})_\infty}{(q)_{\infty}},\\
&\sum_{j_1\geq\cdots\geq j_5\geq 0}
(-1)^{j_3+j_5}
\frac{q^{j_1+j_2-j_4+j_1^2+j_2^2-j_4^2-(j_1+j_2+j_3)+\binom{j_3-j_4}{2} + \binom{j_4-j_5}{2}}}
{(q)_{j_1-j_2}(q)_{j_2-j_3}(q)_{j_3-j_4}(q)_{j_4-j_5}}
\frac{(-1;q^3)_{j_5}}{(q)_{2j_5}(-1;q)_{j_5}}\nonumber\\
&\quad\quad=
\frac{(q^4,q^5,q^9;\,\,q^9)_{\infty}(q^1,q^{17};\,\,q^{18})_\infty}{(q)_{\infty}}.
\end{align}
The triple-sum simplifies, using \eqref{eqn:b1bc1}:
\begin{align}
\sum_{ j_3\geq 0}
q^{j_3^2-j_3}(1-q^{2j_3})
\frac{(-1;q^3)_{j_3}}{(q)_{2j_3}(-1;q)_{j_3}}
=
\sum_{ j_3\geq 0}
q^{j_3^2+j_3}
\frac{(-q^3;q^3)_{j_3}}{(q)_{2j_3+1}(-q;q)_{j_3}}
=
\frac{(q^3,q^6,q^9;\,\,q^9)_{\infty}(q^3,q^{15};\,\,q^{18})_\infty}{(q)_{\infty}}.
\end{align}
The quintuple-sum simplifies after using \eqref{eqn:b1bc1} for the sum over $j_3$ and a bit of algebraic manipulation:
\begin{align}
\sum_{j_1\geq j_4 \geq j_5\geq 0}
(-1)^{j_4+j_5}
\frac{(-q^{j_1^2+2j_1+1}+q^{j_1^2-2j_4})}
{(q)_{j_1-j_4}(q)_{j_4-j_5}}
\frac{(-1;q^3)_{j_5}}{(q)_{2j_5}(-1;q)_{j_5}}
=
\frac{(q^4,q^5,q^9;\,\,q^9)_{\infty}(q^1,q^{17};\,\,q^{18})_\infty}{(q)_{\infty}}.
\end{align}

\subsection{Level 7} Compare the identities here with \cite{Sla-ids} and \cite{TakTsu-A22A132}.
We use the first pair with base $a=q$, moves in Section \ref{sec:lim1} and $k=1$, $i=0,1,2,3$:
\begin{align}
&\sum_{j_1\geq 0}{q^{j_1+j_1^2}}
\frac{1}{(q)_{2j_1}}=
\frac{(q,q^9,q^{10};\,\,q^{10})_{\infty}(q^8,q^{12};\,\,q^{20})_\infty}{(q)_{\infty}},\\
&\sum_{j_1\geq 0}{q^{j_1^2}}
\frac{1}{(q)_{2j_1}}=
\frac{(q^2,q^8,q^{10};\,\,q^{10})_{\infty}(q^6,q^{14};\,\,q^{20})_\infty}{(q)_{\infty}},\\
&\sum_{j_1\geq j_2\geq j_3\geq 0}
(-1)^{j_2+j_3}\frac{q^{j_1+j_1^2-(j_1+j_2)+\binom{j_2-j_3}{2}}}{(q)_{j_1-j_2}(q)_{j_2-j_3}}
\frac{1}{(q)_{2j_3}}
=
\frac{(q^3,q^7,q^{10};\,\,q^{10})_{\infty}(q^4,q^{16};\,\,q^{20})_\infty}{(q)_{\infty}},\\
&\sum_{j_1\geq\cdots\geq j_5\geq 0}
(-1)^{j_3+j_5}
\frac{q^{j_1+j_2-j_4+j_1^2+j_2^2-j_4^2-(j_1+j_2+j_3)+\binom{j_3-j_4}{2} + \binom{j_4-j_5}{2}}}
{(q)_{j_1-j_2}(q)_{j_2-j_3}(q)_{j_3-j_4}(q)_{j_4-j_5}}
\frac{1}{(q)_{2j_5}}\nonumber\\
&\quad\quad=
\frac{(q^4,q^6,q^{10};\,\,q^{10})_{\infty}(q^2,q^{18};\,\,q^{20})_\infty}{(q)_{\infty}}.
\end{align}
The triple-sum simplifies, using \eqref{eqn:b1bc1}:
\begin{align}
\sum_{j_3\geq 0}
q^{j_3^2-j_3}(1-q^{2j_3})
\frac{1}{(q)_{2j_3}}
=
\sum_{j_3\geq 0}
\frac{q^{j_3^2+j_3}}{(q)_{2j_3+1}}
=
\frac{(q^3,q^7,q^{10};\,\,q^{10})_{\infty}(q^4,q^{16};\,\,q^{20})_\infty}{(q)_{\infty}}.
\end{align}
The quintuple-sum simplifies after using \eqref{eqn:b1bc1} for the sum over $j_3$ and a bit of algebraic manipulation:
\begin{align}
\sum_{j_1\geq j_4 \geq j_5\geq 0}
(-1)^{j_4+j_5}
\frac{(-q^{j_1^2+2j_1+1}+q^{j_1^2-2j_4})}
{(q)_{j_1-j_4}(q)_{j_4-j_5}}
\frac{1}{(q)_{2j_5}}
=
\frac{(q^4,q^6,q^{10};\,\,q^{10})_{\infty}(q^2,q^{18};\,\,q^{20})_\infty}{(q)_{\infty}}.
\end{align}

If we instead use the second pair with Section \ref{sec:lim1} and $a=q^2$, $k=1$, $i=0,1,2,3$, we get:
\begin{align}
&\sum_{j_1\geq 0}{q^{2j_1+j_1^2}}
\frac{1}{(q^2;q)_{2j_1}}=
(1-q)\frac{(q^4,q^6,q^{10};\,\,q^{10})_{\infty}(q^2,q^{18};\,\,q^{20})_\infty}{(q)_{\infty}},\\
&\sum_{j_1\geq 0}{q^{j_1+j_1^2}}
\frac{1}{(q^2;q)_{2j_1}}=
(1-q)\frac{(q^3,q^7,q^{10};\,\,q^{10})_{\infty}(q^4,q^{16};\,\,q^{20})_\infty}{(q)_{\infty}},\\
&\sum_{j_1\geq j_2\geq j_3\geq 0}
(-1)^{j_2+j_3}\frac{q^{2j_1+j_1^2-(j_1+j_2)+\binom{j_2-j_3}{2}}}{(q)_{j_1-j_2}(q)_{j_2-j_3}}
\frac{1}{(q^2;q)_{2j_3}}
=
(1-q)\frac{(q^2,q^8,q^{10};\,\,q^{10})_{\infty}(q^6,q^{14};\,\,q^{20})_\infty}{(q)_{\infty}},\\
&\sum_{j_1\geq\cdots\geq j_5\geq 0}
(-1)^{j_3+j_5}
\frac{q^{2j_1+2j_2-2j_4+j_1^2+j_2^2-j_4^2-(j_1+j_2+j_3)+\binom{j_3-j_4}{2} + \binom{j_4-j_5}{2}}}
{(q)_{j_1-j_2}(q)_{j_2-j_3}(q)_{j_3-j_4}(q)_{j_4-j_5}}
\frac{1}{(q^2;q)_{2j_5}}\nonumber\\
&\quad\quad=
(1-q)\frac{(q,q^9,q^{10};\,\,q^{10})_{\infty}(q^8,q^{12};\,\,q^{20})_\infty}{(q)_{\infty}}.
\end{align}
The triple-sum simplifies:
\begin{align}
\sum_{j_3\geq 0}
q^{j_3^2}(1-q^{2j_3+1})
\frac{1}{(q^2;q)_{2j_3}}
= (1-q) + \sum_{j_3\geq 1}
\frac{q^{j_3^2}}{(q^2;q)_{2j_3-1}}
=
(1-q)\frac{(q^2,q^8,q^{10};\,\,q^{10})_{\infty}(q^6,q^{14};\,\,q^{20})_\infty}{(q)_{\infty}}.
\end{align}
The quintuple-sum simplifies after using \eqref{eqn:b1bc1} for the sum over $j_3$ and a bit of algebraic manipulation:
\begin{align}
\sum_{j_1\geq j_4 \geq j_5\geq 0}
(-1)^{j_4+j_5}
\frac{(-q^{j_1^2+3j_1+3}+q^{j_1^2+j_1-2j_4})}
{(q)_{j_1-j_4}(q)_{j_4-j_5}}
\frac{1}{(q^2;q)_{2j_5}}
=
(1-q)\frac{(q,q^9,q^{10};\,\,q^{10})_{\infty}(q^8,q^{12};\,\,q^{20})_\infty}{(q)_{\infty}}.
\end{align}

%
\providecommand{\oldpreprint}[2]{\textsf{arXiv:\mbox{#2}/#1}}\providecommand{\preprint}[2]{\textsf{arXiv:#1
		[\mbox{#2}]}}

\end{document}